\documentclass[11pt]{article}
\usepackage{amsmath, amssymb, amsthm, mathtools}
\usepackage[letterpaper,margin=1in]{geometry}
\usepackage{setspace}
\usepackage{graphicx}

\usepackage{hyperref}

\let\phi=\varphi
\newcommand{\R}{\mathbb{R}}
\newcommand{\Z}{\mathbb{Z}}
\newcommand{\N}{\mathbb{N}}
\newcommand{\C}{\mathbb{C}}
\newcommand{\T}{\mathbb{T}}

\newcommand{\K}{\mathbb{K}}
\newcommand{\cP}{\mathcal{P}}

\newcommand{\bx}{\boldsymbol{x}}

\newcommand{\bu}{\boldsymbol{u}}

\newcommand{\bv}{\boldsymbol{v}}
\newcommand{\bw}{\boldsymbol{w}}

\newcommand{\bc}{\boldsymbol{c}}

\newcommand{\bone}{\boldsymbol{1}}
\newcommand{\bT}{\boldsymbol{T}}

\newcommand{\eps}{\varepsilon}

\newcommand{\Fnorm}[1]{\|#1\|_{F}}
\newcommand{\conj}[1]{\overline{#1}}

\newtheorem{theorem}{Theorem}
\newtheorem{lemma}[theorem]{Lemma}
\newtheorem{proposition}[theorem]{Proposition}
\newtheorem{corollary}[theorem]{Corollary}

\theoremstyle{definition}
\newtheorem{definition}[theorem]{Definition}
\newtheorem{remark}[theorem]{Remark}

\DeclareMathOperator{\Disc}{Disc}
\DeclareMathOperator{\Eig}{Eig}

\title{Low discrepancy and spectral gap  for directed hypergraphs  via  spectral regularity lemma for tensors}
\author{Hi\d{\^e}p H\`an}

\date{\today}

\begin{document}
\maketitle

\begin{abstract}
We show that low discrepancy is equivalent to  spectral gap for directed 
$k$-uniform hypergraphs. For undirected hypergraphs this recovers a theorem of Lenz and Mubayi, with a considerably shorter proof. 
At the heart of our argument is a Frieze--Kannan type spectral regularity lemma for hypergraphs
based on the variational notion of hypergraph eigenvalues  by Friedman--Wigderson, which decomposes
any tensor  into a bounded number of rank-one tensors plus a quasi-random tensor with small top eigenvalue. 
This lemma may be of independent interest, and we prove it for complex-valued, not necessarily symmetric tensors. 

We also briefly discuss a Szemer\'edi-type variant and  the regularization of Cayley-type hypergraphs over finite abelian groups.
\end{abstract}
\leftskip10cm\relax
\rightskip0.3cm\relax
{\it\scriptsize{Thiago, Felix und H\d{a}nh gewidmet  -- f\"ur  eure Zuneigung und euren Rückhalt der letzten 4 Jahre}}

\leftskip0cm\relax
\rightskip0cm\relax

\section{Introduction}

The study of quasi-randomness for graphs and hypergraphs  has been a central topic in discrete mathematics over the last decades.
Loosely speaking, quasi-randomness concerns the study of deterministic properties which are characteristic of random objects, and
for graphs the subject starts with the following notion of uniform edge distribution:
 a sequence $(G_n)_{n\to\infty}$ of graphs $G_n=([n],E_n)$ satisfies $\Disc$ if there is a $p\in[0,1]$ such that 
\begin{equation}
\label{eq:lowdisc}
e(S, T)=p{|S||T|}+o(n^2) \quad\text{for all subsets } S, T\subseteq [n]. 
\end{equation}
From the very beginning, spectral aspects have been at the heart of the area.
Indeed, the cornerstone theorem of Chung, Graham and Wilson~\cite{CGW89} shows that $\Disc$
is  not only \emph{equivalent} to the correct count of every fixed subgraph, but also to
the spectral gap property $\lambda(G_n) = o(\lambda_1(G_n))$ with  $\lambda_1(G_n) \ge \dots \ge \lambda_n(G_n)$ denoting  the
eigenvalues of the adjacency matrix of $G_n$ and
$\lambda(G_n) = \max\{|\lambda_2(G_n)|, |\lambda_n(G_n)|\}$ denotes the
so-called \emph{second eigenvalue}.

Quasi-randomness is inseparable from the regularity method, which
approximates an arbitrary graph by boundedly many quasi-random pieces, classically through Szemer\'edi's
regularity lemma~\cite{Sze78}. Here,
the spectral aspect is prominently represented by the weak
regularity lemma of Frieze and Kannan~\cite{FriezeKannan99}, which, in the language of Tao~\cite{Tao}, decomposes the adjacency matrix $A$ of a graph into $A=S+Q+R$, a sum of a \emph{structured part} $S$ of bounded complexity, a \emph{quasi-random part} $Q$,
and a small
\emph{remainder/error part}~$R$, depending on the spectral accuracy one imposes on~$Q$.
\medskip

The investigation of quasi-randomness has been extended to $k$-uniform hypergraphs, or \emph{$k$-graphs} for short, see for example~\cite{ChungGraham90,ChungGraham91,Chung12,KRS02,KNRS10,CHPS12,LM1,LM2,Towsner17,ACHPS18,Gowers07,RS04,NRS06}. In this setting 
the situation is considerably richer and, after a series of results, Towsner~\cite{Towsner17} established a whole
hierarchy of mutually inequivalent notions of quasi-randomness (see
also~\cite{ACHPS18}). 

In this paper we are interested in the following straightforward generalization
of~\eqref{eq:lowdisc}. A sequence $(H_n)_{n\to\infty}$ of $k$-graphs $H_n=([n],E_n)$ satisfies $\Disc$ if there is a $p\in[0,1]$ such that 
\begin{equation}\label{eq:disc}
   e(S_1, \dots, S_k) = p\,|S_1|\cdots|S_k| + o(n^k)
   \quad\text{for all } S_1, \dots, S_k \subseteq [n],
\end{equation}
where $e(S_1, \dots, S_k)$ denotes the number of ordered $k$-tuples  which span
an edge and with $i$-th coordinate in $S_i$.
For a long time this notion was neglected in the literature
since, unlike the graph case, it does not control the count of copies
of all fixed subhypergraphs once $k \ge 3$. It was only the works of Kohayakawa, Nagle, R\"odl and Schacht~\cite{KNRS10} and
of the author together with Conlon, Person and Schacht~\cite{CHPS12} that established a
Chung--Graham--Wilson type theorem, showing that $\Disc$ equivalent to the correct
count of all fixed \emph{linear} $k$-graphs, i.e., $k$-graphs in
which any two edges share at most one vertex. Later
the spectral side was settled by the remarkable work of Lenz and
Mubayi~\cite{LM1,LM2}, who showed that~\eqref{eq:disc} is equivalent to a notion of spectral gap
based on the variational notion of hypergraph eigenvalue introduced by
Friedman and Wigderson~\cite{FriedmanWigderson95}. It is this relationship that we shall be concerned with here, in the slightly more general setting of a
\emph{directed $k$-graph}, i.e., 
$H = ([n], E)$ where ~$E\subset{[n]}_k$ consists of $k$-tuples of pairwise distinct vertices.
Undirected $k$-graphs  then correspond to directed ones with edge set  invariant under permutation of the
coordinates of its tuples.
Uniform edge distribution for directed $k$-graphs, in its finitary form, is defined in a straightforward manner.
\begin{description}
\item[$\Disc(\delta):$] A  directed $k$-graph $H=([n],E)$ has \emph{discrepancy $\delta$} (``$H$ has $\Disc(\delta)$'')  if there is a $p\in [0,1]$ such that
\[
   \bigl| e(S_1, \ldots, S_k) - p|S_1|\cdots|S_k| \bigr|
   \le \delta n^k
   \qquad \text{for all } S_1, \ldots, S_k \subseteq [n],
\]
\end{description}
where $e(S_1, \ldots, S_k)$ denotes the number of edges
$(i_1, \ldots, i_k) \in E$ with $i_\ell \in S_\ell$ for
all $\ell \in [k]$. 

\medskip

Turning to the spectral side, the role of the adjacency matrix of a graph is
played by the adjacency tensor of a $k$-graph, and that of its eigenvalues by
the variational notion of Friedman and Wigderson~\cite{FriedmanWigderson95},
which we now introduce.
For  $k \ge 2$ and  $\K\in\{\R,\C\}$, a \emph{tensor} on $[n]$
with entries in $\K$ is an array
$\tau = (\tau_{i_1 \cdots i_k}) \in \K^{[n]^k}$. Such a tensor acts as a $k$-linear form by
\[\tau(\bu_1, \ldots, \bu_k) = \sum_{i_1, \ldots, i_k}
\tau_{i_1\cdots i_k}\bu_1(i_1)\cdots\bu_k(i_k)\] 
and we do not distinguish notationally between the tensor and
its form.
The \emph{adjacency tensor} of
a directed $k$-graph $H$ is $\tau_H \in \R^{[n]^k}$ defined by
\[(\tau_H)_{i_1\cdots i_k} = \begin{cases}1 &\text{if } (i_1, \ldots, i_k) \in E,\\ 0&\text{otherwise.}\end{cases}\] 

Note that the adjacency tensor of an {undirected}
$k$-graph is \emph{symmetric}, i.e., it satisfies
$\tau_{i_{\pi(1)} \cdots i_{\pi(k)}} = \tau_{i_1 \cdots i_k}$ 
for every permutation $\pi \in S_k$.

\begin{definition}\label{def:H-eig}
For a tensor $\tau \in \K^{[n]^k}$ its
\emph{top eigenvalue} is
\[\lambda_{\max}(\tau)
   = \sup_{\substack{\bu_1, \ldots, \bu_k \in \K^n\\
   \|\bu_1\| = \cdots = \|\bu_k\| = 1}}
   \bigl| \tau(\bu_1, \ldots, \bu_k) \bigr|.\]

The \emph{top eigenvalue} of  a directed $k$-graph $H$ is $\lambda_1(H)=\lambda_{\max}(\tau_H)$, i.e., the top eigenvalue of its adjacency tensor, and 
the \emph{second eigenvalue} of $H$  is the top eigenvalue of the centered tensor 
\[
   \lambda_2(H) = \lambda_{\max}
   \bigl(\tau_H - p\bone^{\otimes k}\bigr)
\]
where $p = |E|/n^k$ is the edge density of $H$ and  $\bone^{\otimes k}$ is the all-ones tensor.
\end{definition}

For $\K=\R$ this is precisely the notion of eigenvalue introduced  by
Friedman--Wigderson~\cite{FriedmanWigderson95} and utilized by 
Lenz--Mubayi in~\cite{LM1,LM2}. Indeed,  an undirected $k$-graph can be seen as a directed one where each original edge
corresponds to $k!$ directed edges, so that our $p = |E|/n^k$ is
the normalization $k!\,|E(H)|/n^k$ used in their works.
We also note that by a
classical theorem of Banach~\cite{Banach38}  restricting
the supremum to equal arguments
$\bu_1 = \cdots = \bu_k$ does not change the value of 
$\lambda_{\max}$ if the tensor is symmetric.

\begin{description}
\item[$\Eig(\eps):$] A directed $k$-graph $H$ has \emph{eigenvalue separation/spectral gap $\eps$} (``$H$ has $\Eig(\eps)$'') if
\[\lambda_2(H) \le \eps n^{k/2}.\]
\end{description}

It is easy to show that spectral gap implies low discrepancy, which is the content of the following expander mixing lemma.
For undirected $H$ this is the hypergraph expander mixing
lemma of Friedman--Wigderson~\cite{FriedmanWigderson95},
generalized to all partite patterns by
Lenz--Mubayi~\cite{LM1}. 
\begin{proposition}[Expander mixing lemma for tensors]
\label{thm:mixing}
Let $\K \in \{\R, \C\}$, $k \ge 2$ and
$\tau \in \K^{[n]^k}$. Then for all
$\bx_1, \ldots, \bx_k \in \K^n$
\[
   |\tau(\bx_1, \ldots, \bx_k)|
   \le \lambda_{\max}(\tau)\|\bx_1\|\cdots\|\bx_k\|.
\]
In particular, a  directed $k$-graph $H$ with
edge density $p$ satisfies
\[
   \bigl| e(S_1, \ldots, S_k) - p|S_1|\cdots|S_k| \bigr|
   \le \lambda_2(H)\sqrt{|S_1|\cdots|S_k|}
   \qquad \text{for all } S_1, \ldots, S_k \subseteq [n],
\]
and in particular, $\Eig(\eps)$ implies $\Disc(\eps)$.
\end{proposition}

\begin{proof}
If some $\bx_\ell = 0$ both sides vanish, so assume all
$\bx_\ell \ne 0$. By multilinearity and the definition of
$\lambda_{\max}$,
\[
   |\tau(\bx_1, \ldots, \bx_k)|
   = \|\bx_1\|\cdots\|\bx_k\|
   \Bigl|\tau\Bigl(\tfrac{\bx_1}{\|\bx_1\|}, \ldots,
   \tfrac{\bx_k}{\|\bx_k\|}\Bigr)\Bigr|
   \le \lambda_{\max}(\tau)\|\bx_1\|\cdots\|\bx_k\|.
\]
For 
the second part let $\sigma = \tau_H - p\bone^{\otimes k}$ be the centered tensor and assume that $\Eig(\eps)$ holds, i.e., $\lambda_{\max}(\sigma) = \lambda_2(H)\leq \eps n^{k/2}.$
Applying the first part to $\sigma$ and
$\bx_\ell = \bone_{S_\ell}$ with $\|\bone_{S_\ell}\| = \sqrt{|S_\ell|} \le \sqrt n$ we obtain
\[\big|e(S_1, \ldots, S_k) - p|S_1|\cdots|S_k|\big|=|\sigma(\bone_{S_1}, \ldots, \bone_{S_k})|\leq \lambda_{\max}(\sigma) n^{k/2}\leq\eps n^k\]
which shows that $H$ satisfies
$\Disc(\eps)$.
\end{proof}

Our main result shows the converse.
\begin{theorem}
\label{thm:main}
For every $k \ge 2$ and $\eps >0$ there is a
 $\delta > 0$ such that every directed
$k$-graph~$H$ satisfying $\Disc(\delta)$ also satisfies $\Eig(\eps)$. 
\end{theorem}
Together with Proposition~\ref{thm:mixing} this shows that the two notions $\Disc$ and $\Eig$ are equivalent.

\begin{corollary}\label{cor:disc-eig}
A sequence $(H_n)_{n\to\infty}$ of directed $k$-graphs satisfies
$\Disc$ if and only if it satisfies $\Eig$, that is,  
 \[e(S_1, \ldots, S_k) = p|S_1|\cdots|S_k|+ o(\delta n^k)
   \quad \forall S_1, \ldots, S_k \subseteq [n]\quad\text{if and only if}\quad 
\lambda_2(H_n) = o(n^{k/2}).\]
\end{corollary}

For undirected $k$-graphs, i.e., those whose edge set is invariant under permuting coordinates, Theorem~\ref{thm:main} is the theorem of Lenz and Mubayi and probably their proof could be adapted to accommodate the directed case as well. 
With that said, their proof was quite involved, and moreover, establishing the result first only for codegree-regular  $k$-graphs in~\cite{LM1} and then for general ones in~\cite{LM2}. 
In this paper we give a shorter 
proof of Theorem~\ref{thm:main} and our setting naturally accommodates directed hypergraphs at essentially no additional technical or notational cost. We also note that there are natural examples of interesting directed $k$-graphs, among them the Cayley-type $k$-graphs over finite abelian groups considered in Section~\ref{sec:concluding}. 

Our argument utilizes   Theorem~\ref{thm:reglem} from below,  a spectral regularity lemma for hypergraphs in the spirit of the graph case, stating that tensors can be decomposed into a structured  tensor (with low rank) and 
a quasi-random tensor (with small top eigenvalue).
This might be of independent interest and
we  prove it in the  slightly more general form of complex valued tensors, noting that we could have restricted ourselves to real valued tensors for Theorem~\ref{thm:main}.  
We hope that the more general version will find further applications and connections and in Section~\ref{sec:concluding} we 
briefly discuss variants of the lemma and further,  how group characters naturally show up when regularizing Cayley hypergraphs, making $\K=\C$ the natural choice.

\medskip 

One obstacle to  a spectral regularity lemma for $k$-graphs is that, in contrast to the matrix case, tensors do not permit an orthogonal singular/eigenvalue decomposition. Indeed,
Cartwright--Sturmfels~\cite{CartwrightSturmfels13} show that a generic symmetric $3$-tensor on $\C^n$ has $2^n - 1$ complex
projective eigenvectors, with no distinguished orthonormal
subfamily. Nevertheless, we show that the natural process of repeatedly peeling off   rank one tensors defined by  top eigenpairs must terminate and yields a Frieze--Kannan type 
decomposition (with no orthogonality).

For tensors $\sigma$ and $\tau$  the \emph{Frobenius inner product} and the \emph{Frobenius norm} are given by
\[
   \langle \sigma, \tau\rangle_F =
   \sum_{i_1, \ldots, i_k \in [n]}
   \sigma_{i_1 \cdots i_k}\conj\tau_{i_1 \cdots i_k}
 \qquad  \qquad\text{and}\qquad\qquad
   \Fnorm{\tau} = \sqrt{\langle \tau, \tau\rangle_F}.
\]
For $\bu_1,\dots, \bu_k \in \K^n$  the \emph{rank one} tensor
$\bu_1\otimes\dots\otimes \bu_k $  is defined by 
\[(\bu_1\otimes\dots\otimes \bu_k)_{i_1 \cdots i_k}
= \bu_1(i_1)\cdots\bu_k(i_k)\]
and we write $\bu^{\otimes k}$ when  $\bu_i=\bu$ for all $i\in[k]$.
A tuple $(\lambda; \bu_1, \ldots, \bu_k)$ with unit
vectors $\bu_i$'s,
$\lambda = \tau(\conj\bu_1, \ldots, \conj\bu_k)$ and
$|\lambda| = \lambda_{\max}(\tau)$ is called a
\emph{top eigenpair} of $\tau$.
Our regularity lemma then reads as follows.
\begin{theorem}[Spectral regularity lemma for tensors]
\label{thm:reglem}
Let $\K \in \{\R, \C\}$ and $k\geq 2$.   For any   $\eps > 0$ and any tensor
$\tau \in \K^{[n]^k}$ there exists   a
$K \le 1/\eps^2$ and  a decomposition
\begin{equation}\label{eq:decomp}
   \tau = \sum_{j=1}^{K}\lambda_j (\bu_{j,1}\otimes\dots\otimes \bu_{j,k})
   + Q \qquad \text{with}\qquad  \lambda_{\max}(Q) \le \eps\Fnorm{\tau},
\end{equation}
unit vectors $( \bu_{j,1},\cdots,\bu_{j,k})$ and  
$\lambda_1, \ldots, \lambda_K\in\R$ satisfying
$\lambda_j > \eps\Fnorm{\tau}$, $j \in [K]$.
In this decomposition each $(\lambda_j; \bu_{j,1}, \ldots, \bu_{j,k})$ may be
chosen to be a top eigenpair of the residual tensor
$\tau -\sum_{i=1}^{j-1}\lambda_i(\bu_{i,1}\otimes\cdots\otimes\bu_{i,k})$.

\medskip
Moreover, for a symmetric $\tau$ we may choose the  factors all equal $\bu_{j,1}=\dots=\bu_{j,k}$, $j\in[K]$, so that the rank-one tensors are $\bu_j^{\otimes k}$,   at the cost that for $\K=\R$ and $k$ even the condition on the $\lambda_j$'s
 being relaxed to
$|\lambda_j| > \eps\Fnorm{\tau}$, $j \in [K]$.
\end{theorem}

Similar to the terminology of Tao~\cite{Tao} we call the first part of the decomposition the \emph{structured part} and the second the \emph{quasi-random part}. Analogous to  the  matrix case the structured part can be used to   partition the ground set $[n]$ 
(by  joint approximate level sets of the $\bu_{j,\ell}$'s) and approximate the original tensor by a block-constant, bounded complexity  tensor, see Lemma~\ref{lem:partition} from below. Further,
the proof of Theorem~\ref{thm:reglem} can be coupled with a growth function to give a (truly) Frieze--Kannan or Szemer\'edi
 type (weak) regularity lemma for hypergraphs, with stronger bounds on the quasi-random part, but at the cost of an additional \emph{residual/error part} in the decomposition, see Section~\ref{sec:concluding}.

\section{A spectral regularity lemma for tensors}
\label{sec:reglempart}

In this section we prove the spectral regularity lemma, Theorem~\ref{thm:reglem}. 
Throughout, let $\K\in\{ \R,\C\}$ and 
equip
$L^2([n]) = \K^n$  with the standard inner product
$\langle \bu, \bv\rangle = \sum_{x \in [n]}\bu(x)\conj\bv(x)$
and the induced norm $\|\bu\| = \sqrt{\langle\bu,\bu\rangle}$. 
We then have the following useful relation to the Frobenius inner product
\begin{align}
\begin{aligned}\label{eq:FrobvsSt}
   \langle \bu_1{\otimes }\dots\otimes \bu_k,\bv_1{\otimes }\dots\otimes\bv_k\rangle_F
   &= \sum_{i_1, \ldots, i_k}
     \bu_1(i_1)\cdots\bu_k(i_k)
     \conj\bv_1(i_1)\cdots\conj\bv_k(i_k)\\
 &  = \prod_{j\in[k]}\sum_i \bu_j(i)\conj\bv_j(i)
   = \prod_{j\in[k]}\langle\bu_j,\bv_j\rangle
   \end{aligned}
\end{align}
and in particular, $\Fnorm{\bu^{\otimes k}}^2 = \|\bu\|^{2k}$.

Further,  the Frobenius inner product of $\tau$  with a rank-one tensor (essentially) converts to evaluation 
\begin{align}\label{eq:evalFrob}
\langle \tau, \bu_1{\otimes }\dots\otimes\bu_k\rangle_F
 = \sum \tau_{i_1\cdots i_k}\conj\bu_1(i_1)\cdots\conj\bu_k(i_k)
 = \tau(\conj\bu_1, \ldots, \conj\bu_k).
\end{align} 
In particular, by Cauchy-Schwarz inequality and~\eqref{eq:FrobvsSt} we have
\begin{align}\label{eq:CSI}
\tau(\bu_1, \ldots, \bu_k)=\langle \tau, \conj\bu_1{\otimes }\dots\otimes\conj\bu_k\rangle_F\leq \Fnorm{\tau}\Fnorm{\conj\bu_1{\otimes }\dots\otimes\conj\bu_k}= \Fnorm{\tau}\prod_{j}\|\bu_j\|.
\end{align}
 The Pythagoras identity in the following form is the heart of the
framework.
\begin{lemma}[Pythagoras identity]\label{lem:Pythagoras}
Let $\K\in\{ \R,\C\}$. For a tensor
$\tau \in \K^{[n]^k}$ and unit vectors
$\bu_1, \ldots, \bu_k \in \K^n$ let
$\lambda = \tau(\conj\bu_1, \ldots, \conj\bu_k)
\in \K$. Then
we have  
\begin{equation}\label{eq:pythagoras}
   \Fnorm{\tau - \lambda(\bu_1\otimes\cdots\otimes\bu_k)}^2 =
   \Fnorm{\tau}^2 - |\lambda|^2. 
\end{equation}

In particular, let  $\tau_0 = \tau$  and for a sequence
of rank-one tensors
$(\bu_{j,1}\otimes\cdots\otimes\bu_{j,k})$,
$j \ge 1$, with unit factors define for  $j\geq1$ 
\[
   \lambda_j =
   \tau_{j-1}(\conj\bu_{j,1}, \ldots,
   \conj\bu_{j,k})  \in \K\qquad\text{and}
   \qquad
   \tau_j = \tau_{j-1} - \lambda_j(\bu_{j,1}\otimes\cdots\otimes\bu_{j,k}).
\]
Then for every $N \ge 1$,
\begin{equation}\label{eq:peel-drop-identity}
   \Fnorm{\tau_N}^2 = \Fnorm{\tau}^2
    - \sum_{j=1}^N |\lambda_j|^2\qquad \text{and thus}\qquad
\sum_{j=1}^N |\lambda_j|^2 \le \Fnorm{\tau}^2.
\end{equation}

\end{lemma}

\begin{proof}
Direct expansion yields \[
   \Fnorm{\tau - \lambda (\bu_1\otimes\cdots\otimes\bu_k)}^2
   = \Fnorm{\tau}^2
     - 2\mathrm{Re}\langle\tau,
       \lambda (\bu_1\otimes\cdots\otimes\bu_k)\rangle_F
     + |\lambda|^2\Fnorm{\bu_1\otimes\cdots\otimes\bu_k}^2.
\]
 By \eqref{eq:FrobvsSt} we have $\Fnorm{\bu_1\otimes\cdots\otimes\bu_k}^2 = \prod_i\|\bu_i\|^{2} = 1$ and by \eqref{eq:evalFrob} 
\[
 \langle\tau, \lambda(\bu_1\otimes\cdots\otimes\bu_k)\rangle_F
 = \conj\lambda\langle\tau, \bu_1\otimes\cdots\otimes\bu_k\rangle_F
 = \conj\lambda\tau(\conj\bu_1, \ldots, \conj\bu_k)
 = |\lambda|^2\]
which yields~\eqref{eq:pythagoras}.

Applying~\eqref{eq:pythagoras}  to $\tau_{j-1}$ with $\bu_\ell = \bu_{j,\ell}$ and
$\lambda = \lambda_j$ yields
$\Fnorm{\tau_j}^2 = \Fnorm{\tau_{j-1}}^2 - |\lambda_j|^2$ and
telescoping over $j = 1, \ldots, N$ gives
\eqref{eq:peel-drop-identity}. 
\end{proof}

We are now in the position to prove the spectral regularity lemma, Theorem~\ref{thm:reglem}.
\begin{proof}[Proof of Theorem~\ref{thm:reglem}]
For given $\eps > 0$ and $\tau$ we   construct a sequence of rank-one tensors
$\bu_{j,1}\otimes\cdots\otimes\bu_{j,k}$ with unit
factors as follows. 
 
 Let
$\tau_0 = \tau$ and suppose $\tau_{j-1}$ has been defined for
some $j \ge 1$.
If $\lambda_{\max}(\tau_{j-1}) \le \eps\Fnorm{\tau}$, we
set $K = j - 1$, $Q= \tau_{j-1}$, and terminate. Otherwise,  choose 
$(\lambda_j;\bu_{j,1}, \ldots, \bu_{j,k})$ to be the top eigenpair of $\tau_{j-1}$, i.e.,
$\bu_{j,1}, \ldots, \bu_{j,k}$ are unit vectors with
$
   \lambda_{\max}(\tau_{j-1}) =
   |\tau_{j-1}(\conj\bu_{j,1},\dots,\conj\bu_{j,k})|
   > \eps\Fnorm{\tau}$ and $\lambda_j=\tau_{j-1}(\conj\bu_{j,1},\dots,\conj\bu_{j,k})$. Then
$|\lambda_j| > \eps\Fnorm{\tau}$ by construction and we set $\tau_j=\tau_{j-1} - \lambda_j(\bu_{j,1}\otimes\cdots\otimes\bu_{j,k})$ and 
repeat with $j$ replaced by $j+1$.

If $\tau$ is symmetric then by Banach's
theorem~\cite{Banach38} the supremum defining
$\lambda_{\max}(\tau_{j-1})$ is attained at equal
arguments $\bu_{j,\ell} = \bu_j$  and choosing such  maximizers keeps every residual tensor $\tau_j$ symmetric.
Moreover, over $\K=\C$ the rotation $\bu_{j,\ell}\mapsto e^{i\theta}\bu_{j,\ell}$ with $\theta = \arg(\lambda_j)$ 
multiplies $\lambda_j$ by $e^{-i\theta}$ and over $\R$ the flip
$\bu_{j,\ell} \mapsto -\bu_{j,\ell}$ multiplies $\lambda_j$ by
$(-1)^k$,   leaving the tensor
 $\lambda_j(\bu_{j,1}\otimes\cdots\otimes\bu_{j,k})$
unchanged in both cases. 
Thus, for arbitrary tensor we may choose $\lambda_j$'s to be positive reals by applying these operations to one factor.
In the case  of symmetric tensors, and if we were to preserve $\bu_{j,\ell} = \bu_j$, then  these operations  have to be applied to all factors. This still allow us 
to choose $\lambda_j$'s  positive  for $\K=\C$ and for $\K=\R$ and $k$ odd.

We apply Lemma~\ref{lem:Pythagoras} to the sequence $(\lambda_j;\bu_{j,1}, \ldots, \bu_{j,k})_j$ and $\tau_j$ and conclude $\sum_{j=1}^\ell |\lambda_j|^2 \le \Fnorm{\tau}^2$ for any $\ell$. Combined
with $|\lambda_j|^2 > \eps^2\Fnorm{\tau}^2$ this forces
$\ell\eps^2\Fnorm{\tau}^2 < \Fnorm{\tau}^2$. Thus
$\ell < 1/\eps^2$ and the process terminates after at most
$K\leq 1/\eps^2$ steps.
Upon termination, telescoping
 gives
\[\tau = \sum_{j=1}^K \lambda_j(\bu_{j,1}\otimes\cdots\otimes\bu_{j,k}) + Q\qquad \text{with each}\quad |\lambda_j| > \eps\Fnorm{\tau}\qquad \text{and}\qquad     \lambda_{\max}(Q) \le \eps\Fnorm{\tau}
\]
as desired.  
\end{proof}

As mentioned previously and analogous to the matrix case the structured part can be used to  
 approximate the original tensor $\tau$ by a block-constant, bounded complexity  tensor via a partition of the ground set.
 More precisely, for a partition $\cP=(V_i)_{i\in[t]}$ of a subset of $V\subset [n]$ and a tensor $\pi\in\K^{[n]^k}$  the \emph{block-averaging  of $\pi$ over $\cP$}, denoted
 $\pi_\cP$, is the tensor taking the constant value ${\pi(\bone_{V_{i_1}},\dots,\bone_{V_{i_k}})}/\prod_{j\in[k]}{|V_{i_j}|}$ on the block  $(V_{i_1},\dots,{V_{i_k}})\in\cP^k$  and is zero on tuples meeting $[n]\setminus V$.
\begin{lemma}[Partition from structured part]
\label{lem:partition}
Let $C \ge 1$, $\alpha \in (0, 1]$. For unit vectors $(\bu_{j,\ell})_{j,\ell}$ in $\K^n$ and reals 
$(\lambda_j)_j$ let
\[S = \sum_{j=1}^{K}\lambda_j(\bu_{j,1}\otimes\cdots\otimes\bu_{j,k}).\]
Then there exist a set $B \subseteq [n]$ of size
$|B| \le \frac{kKn}{C^2}$
and a partition $\mathcal P=(V_i)_{i\in[t]}$ of $[n] \setminus B$ into at most $(2C/\alpha + 1)^{2kK}$ cells such that 
 the block-averaging   $S_{\cP}$  of $S$ over  $\cP$ satisfies    
 \[
      \bigl| S(\bx_1, \ldots, \bx_k)
      - S_{\mathcal P}(\bx_1, \ldots, \bx_k) \bigr|
      \le 2k\alpha C^{k-1} \sum_{j=1}^{K} |\lambda_j| 
   \]
  for all vectors $\bx_1, \ldots, \bx_k$ supported on
$[n]\setminus B$ with $\|\bx_\ell\| \le 1$, $\ell\in[k]$.
\end{lemma}

\begin{proof}
The proof is analogous to the matrix case, $k=2$. We define \[B_{j,\ell} = \{v \colon |\bu_{j,\ell}(v)| > C/\sqrt n\}\qquad \text{and}\qquad
B = \bigcup_{j \le K,\ \ell \le k} B_{j,\ell}.\] As the $\bu_{j,\ell}$'s are unit vectors we have
$|B_{j,\ell}| \le n/C^2$ and $|B| \le kKn/C^2$, as claimed. 
Let
$\hat\bu_{j,\ell}$ be the vector~$\bu_{j,\ell}$ restricted on $[n]\setminus{B_{j,\ell}}$ so that
$\|\hat\bu_{j,\ell}\|_\infty \le C/\sqrt n$.

For each $j\leq K$ and $\ell \le k$ define the level sets, for $r \in \Z$, 
\[L_{j,\ell}(r)=\left\{i\in[n] \colon \hat\bu_{j,\ell}(i) \in \left[r\tfrac{\alpha}{\sqrt n},
(r{+}1)\tfrac{\alpha}{\sqrt n}\right)\right\}\]
and over $\C$ we do this
for real and imaginary parts separately. We partition $[n]\setminus B$ by the joint level sets $L_{j,\ell}(r)$ over all $j, \ell$ and $r$. Since the 
values lie in $[-C/\sqrt n, C/\sqrt n]$, each pair $(j,\ell)$
contributes at most  $(2C/\alpha + 1)^2$ levels giving at most
$(2C/\alpha + 1)^{2kK}$ cells in total.

Note that any two vertices $i, i'$ in the same cell satisfy $|\hat\bu_{j,\ell}(i) - \hat\bu_{j,\ell}(i')| \le \sqrt2\alpha/\sqrt n$  for
every $j$ and $\ell$,
thus any two
tuples $(i_1, \ldots, i_k),(i_1', \ldots, i_k')\in U_1 \times \cdots \times U_k$ of the same box satisfies 
\[
   \bigl|\hat\bu_{j,1}(i_1)\cdots\hat\bu_{j,k}(i_k)
   - \hat\bu_{j,1}(i_1')\cdots\hat\bu_{j,k}(i_k')\bigr|
   \le k\Bigl(\frac{C}{\sqrt n}\Bigr)^{k-1}
   \frac{\sqrt2\alpha}{\sqrt n}.
\]
Hence, by the triangle inequality  
\[
   \bigl|S_{i_1\cdots i_k} - S_{i_1'\cdots i_k'}\bigr|
   \le  \frac{\sqrt2k\alpha C^{k-1}}{n^{k/2}}
   \sum_{j=1}^{K}|\lambda_j|=:\frac{\delta}{n^{k/2}},
\]
i.e., all entries of $S$ agree up to $\delta/n^{k/2}$ over a box.
By definition $(S_{\cP})_{i_1\cdots i_k}$ is  the
average of the entries of $S$ over the box containing
$(i_1, \ldots, i_k)$, an average of numbers all within
$\delta/n^{k/2}$, thus
\[
   \bigl|(S - S_{\cP})_{i_1\cdots i_k}\bigr| \le \delta
   \qquad \text{for all } i_1, \ldots, i_k \in
   [n]\setminus B .
\]
Finally, let $\bx_1, \ldots, \bx_k$ be vectors supported on
$[n]\setminus B$ with $\|\bx_\ell\| \le 1$, $\ell\in[k]$. By Cauchy--Schwarz we have
$\|\bx_\ell\|_1 \le \sqrt n$, so
\[
   \bigl|S(\bx_1, \ldots, \bx_k)
   - S_{\cP}(\bx_1, \ldots, \bx_k)\bigr|
   \le \frac{\delta}{n^{k/2}} \sum_{i_1, \ldots, i_k}
   |\bx_1(i_1)|\cdots|\bx_k(i_k)|
   = \frac{\delta}{n^{k/2}}\prod_{\ell\in[k]}\|\bx_\ell\|_1 \leq \delta
\]
as claimed. 
\end{proof}

\section{From low discrepancy to spectral gap}
\label{sec:regproof}

In this section we prove Theorem~\ref{thm:main} and we will, as mentioned before,  work entirely on $\K=\R$.

\begin{proof}[Proof of Theorem~\ref{thm:main}]
Given $k \ge 2$ and $\eps \in (0, 1)$ let 
\[C = \frac{288\cdot 2^k\sqrt k}{\eps^2}, \qquad    \alpha = \frac{\eps}{288kC^{k}},\qquad    T= \Bigl(\frac{2C}{\alpha} + 1\Bigr)^{288k/\eps^2}\qquad\text{and}\qquad \delta =  \Bigl(\frac{\eps}{24\cdot 2^kT}\Bigr)^{2k}.
\]
For a contradiction suppose that  there is a  directed $k$-graph $H$ with edge density $p$ that satisfies $\Disc(\delta)$ but  fails $\Eig(\eps)$, i.e.,
$\lambda_2(H) = \lambda_{\max}(\sigma) > \eps n^{k/2}$. By applying $\Disc(\delta)$ to $S_1 = \cdots = S_k = [n]$ we have 
$| e(S_1, \ldots, S_k) - p|S_1|\cdots|S_k| | \le 2\delta n^k$ for all $S_1, \ldots, S_k \subseteq [n]$.

Let 
$\sigma = \tau_H - p\bone^{\otimes k}$ be the centered adjacency tensor of $H$ and note that
$\Fnorm{\sigma} \le n^{k/2}$, since  entries of
$\sigma$ have modulus at most $1$. We apply Theorem~\ref{thm:reglem} to $\sigma$ and with the parameter $\eps/12$ and obtain a $K<144/\eps^2$
and a decomposition \[\sigma = S + Q\qquad\text{with} \qquad
S = \sum_{j=1}^K\lambda_j(\bu_{j,1}\otimes\cdots\otimes\bu_{j,k})\qquad \text{and}\qquad\lambda_{\max}(Q) \le \frac\eps{12} \Fnorm{\sigma},\]
with the $(\lambda_j; \bu_{j,1}, \ldots, \bu_{j,k})$
chosen as top eigenpairs of the successive residuals.
Lemma~\ref{lem:Pythagoras} also gives
\begin{align}\label{eq:K0}\sum_{j\leq K}|\lambda_j|^2 \le \Fnorm{\sigma}^2\qquad \text{and hence}\qquad
\sum_{j\leq K}|\lambda_j| \le \sqrt{K}\Fnorm{\sigma}\end{align} by
Cauchy--Schwarz. 
Since
$\lambda_{\max}(\sigma) > \eps n^{k/2} >
\frac\eps{12}\Fnorm{\sigma} \ge \lambda_{\max}(Q)$, the
structured part $S$ is nontrivial, i.e., $K \ge 1$ and  we
have
\[
   |\sigma(\bu_{1,1}, \ldots, \bu_{1,k})| = |\lambda_1|
   = \lambda_{\max}(\sigma) > \eps n^{k/2}.
\]
By applying Lemma~\ref{lem:partition} to $S$ with the parameters
$C$ and $\alpha$ we obtain  a set $B$ with
$|B| \le kK n/C^2$, a partition
$\cP'$ of $[n]\setminus B$ into $t' \le T$
cells. We move
 cells of size less than $\delta^{1/k}n$ into a set $U$, set
$J = B \cup U$ and obtain $|J| \le kK n/C^2 + T\delta^{1/k}n \le \tfrac{\eps^2}{288\cdot 4^k}n$.
Let the remaining partition be denoted by $\cP = (V_i)_{i \in [t]}$, $t\leq T$ and let $S_{\mathcal P}$ be the block-averaging tensor of $S$.

For each $\ell \in [k]$ we split
$\bu_{1,\ell} = \bv_\ell + \bv_{\ell,J}$ where
$\bv_{\ell,J}$ is the vector  $\bu_{1,\ell}$ restricted on ${J}$. Then
$\sigma(\bu_{1,1}, \ldots, \bu_{1,k})=\sigma(\bv_1,\dots, \bv_k)+\sum \sigma(\bw_1, \ldots, \bw_k)$ where in each  $\sigma(\bw_1, \ldots, \bw_k)$ we have $\bw_\ell \in \{\bv_\ell, \bv_{\ell,J}\}$ and there is some $\ell$ so that $\bw_\ell=\bv_{\ell,J}$. 
Fix such a tuple $(\bw_1, \ldots, \bw_k)$ and let $\bT=T_1 \times \cdots \times T_k$ with  $T_i$ being the support of $\bw_i$.
Clearly,  $\sigma(\bw_1, \ldots, \bw_k)=\sigma|_{\bT}(\bw_1, \ldots, \bw_k)$ and since $T_\ell\subset J$ for some $\ell$
we obtain by Cauchy--Schwarz for the Frobenius inner product~\eqref{eq:CSI} and by~\eqref{eq:FrobvsSt}  
\[|\sigma(\bw_1, \ldots, \bw_k)|=\big|\sigma|_{\bT}(\bw_1, \ldots, \bw_k)\big|=\langle\sigma|_{\bT},\bw_1\otimes \ldots\otimes \bw_k\rangle\leq \Fnorm{\sigma|_{\bT}} \prod_\ell\|\bw_{\ell}\|\leq \Fnorm{\sigma|_{\bT}}\leq (|J|n^{k-1})^{1/2}\]
from which we conclude that
\[
   |\sigma(\bv_1, \ldots, \bv_k)|
   \ge |\sigma(\bu_{1,1}, \ldots, \bu_{1,k})|
   - (2^k-1)(|J|n^{k-1})^{1/2}
   > \tfrac{11}{12}\eps n^{k/2}.
\]
Let $\sigma_{\mathcal P}=S_\cP+Q_\cP$
be the  block-averaging of $\sigma=S+Q$  over $\cP$ and we now show  that
\begin{align}\label{eq:sigmaP}
\bigl|\sigma(\bv_1, \ldots, \bv_k)
   - \sigma_{\mathcal P}(\bv_1, \ldots, \bv_k)\bigr|\leq  \tfrac14\eps n^{k/2} \qquad\text{which implies}\qquad |\sigma_{\mathcal P}(\bv_1, \ldots, \bv_k)| > \tfrac23\eps n^{k/2}.
\end{align}
 As $\|\bv_\ell\| \le 1$ for every $\ell$, Lemma~\ref{lem:partition} together with the choice of $\alpha$ and~\eqref{eq:K0} immediately gives
\begin{align}\label{eq:SSP}|S(\bv_1, \ldots, \bv_k) -
S_{\mathcal P}(\bv_1, \ldots, \bv_k)| \le
2k\alpha C^{k-1}\sum_j|\lambda_j| \le
2k\alpha C^{k-1}\sqrt{K}\Fnorm{\sigma} \leq \tfrac{\eps}{12}\Fnorm{\sigma}.\end{align}

To derive an analogous bound for $Q$ and $Q_\cP$ recall that $Q_\cP$ takes the  value
$\frac{Q(\bone_{V_{i_1}}, \ldots,\bone_{V_{i_k}})}{(|V_{i_1}|\cdots|V_{i_k}|)}$  on the block $V_{i_1} \times \cdots \times V_{i_k}\in\cP^k$. Thus, writing $\nu_{\ell,i} = \sum_{x \in V_i}\bv_\ell(x)$ 
and 
by organizing the sum by blocks and using multi-linearity
we obtain
\begin{align*}Q_{\mathcal P}(\bv_1, \ldots, \bv_k)&=\sum_{i_1 \ldots i_k \in [t]^k}
\frac{Q(\bone_{V_{i_1}}, \ldots,\bone_{V_{i_k}})}{|V_{i_1}|\cdots|V_{i_k}|} \nu_{1,i_1}\dots\nu_{k,i_k}\\
&=\sum_{i_1\dots i_k\in[t]^k}  Q\left(\frac{\nu_{1,i_1}}{|V_{i_1}|}\bone_{V_{i_1}}, \ldots,\frac{\nu_{k,i_k}}{|V_{i_k}|}\bone_{V_{i_k}}\right) =Q(\bv_{1,\cP}, \ldots, \bv_{k,\cP}).
\end{align*}
Further, $\|\bv_{\ell,\cP}\|^2=\sum_{i\in[t]}|V_i|\big(\nu_{\ell,i}/|V_i|\big)^2\leq \sum_{i\in[t]}\sum_{x\in V_i} |\bv_\ell(x)|^2= \|\bv_\ell\|^2\leq 1$ for each $\ell$, thus 
\[|Q(\bv_1, \ldots, \bv_k)- Q_{\mathcal P}(\bv_1, \ldots, \bv_k)| \leq |Q(\bv_1,\dots,\bv_k)|+ |Q(\bv_{1,\cP}, \ldots, \bv_{k,\cP})|\leq 2 \lambda_{\max}(Q) \leq \frac\eps6\Fnorm{\sigma}\]
and together with~\eqref{eq:SSP} and $\sigma-\sigma_\cP=(S-S_\cP)+(Q-Q_\cP)$ we conclude~\eqref{eq:sigmaP}.

Finally, as  $\sum_{i\in[t]}\frac{|\nu_{\ell,i}|}{\sqrt{|V_i|}}\leq\sqrt t \left(\sum_{i\in[t]}\frac{(\nu_{\ell,i})^2}{|V_i|}\right)^{1/2}= \sqrt t \|\bv_{\ell,\cP}\|\leq \sqrt T$ for each $\ell$ and $|V_i|>\delta^{1/k} n$ for each $i\in[t]$ we conclude from~\eqref{eq:sigmaP}
\begin{align*}
\tfrac23\eps n^{k/2}<|\sigma_{\mathcal P}(\bv_1, \ldots, \bv_k)|&\leq\sum_{i_1 \ldots i_k \in [t]^k}
\frac{|\sigma(\bone_{V_{i_1}}, \ldots,\bone_{V_{i_k}})|}{|V_{i_1}|\cdots|V_{i_k}|} |\nu_{1,i_1}|\dots|\nu_{k,i_k}|\\
&=\sum_{i_1 \ldots i_k \in [t]^k} \frac{\bigl|\sigma(\bone_{V_{i_1}}, \ldots,
   \bone_{V_{i_k}})\bigr|}{\prod_{\ell\in[k]}\sqrt{|V_{i_\ell}|}}\prod_{\ell\in[k]}\frac{|\nu_{\ell,i_\ell}|}{\sqrt{|V_{i_\ell}|}}\\
   &\leq \max_{i_1 \ldots i_k\in[t]^k}
   \frac{\bigl|\sigma(\bone_{V_{i_1}}, \ldots,
   \bone_{V_{i_k}})\bigr|}{\prod_{\ell\in[k]}\sqrt{|V_{i_\ell}|}}
   \cdot\prod_{\ell\in[k]}\Bigl(\sum_{i\in[t]}\tfrac{|\nu_{\ell,i}|}{\sqrt{|V_i|}}\Bigr)\\
   &\leq \frac{T^{k/2}}{\delta^{1/2}n^{k/2}}
   \max_{i_1 \ldots i_k\in[t]^k}
   \bigl|\sigma(\bone_{V_{i_1}}, \ldots,
   \bone_{V_{i_k}})\bigr|.
\end{align*}
Thus, there is a tuple ${i_1}\dots {i_k}\in[t]^k$ with
\[
   \bigl| e(V_{i_1}, \ldots, V_{i_k})
   - p|V_{i_1}|\cdots|V_{i_k}| \bigr|
   = \bigl|\sigma(\bone_{V_{i_1}}, \ldots,
   \bone_{V_{i_k}})\bigr|
   \ge \frac{2\eps\delta^{1/2}}{3T^{k/2}}n^k
   > 2\delta n^k,
\]
which contradicts the fact that $H$ satisfies $\Disc(\delta)$.
\end{proof}

\section{Concluding remarks}\label{sec:concluding}
\paragraph{Variants of the spectral regularity lemma}
We briefly show that the mechanism used to prove the spectral regularity lemma, Theorem~\ref{thm:reglem}, can be coupled with a growth function to 
show a (truly) Frieze--Kannan or Szemer\'edi
 type (weak) regularity lemma for hypergraphs, with stronger bounds for the quasi-random part,  but at the cost of an additional {residual/error part} in the decomposition.
By taking the growth function $F$ in the following to be  exponential  or  tower-type one obtains Frieze--Kannan type and Szemer\'edi-type regularity lemma, respectively.

\begin{theorem}[Spectral regularity lemma with growth function]
\label{thm:tao-reglem}
Let $\K \in \{\R,\C\}$, $k \ge 2$, and let
$F \colon\N \to \N$ be a strictly increasing function. For
every $\eps > 0$ and every
tensor $\tau \in \K^{[n]^k}$ there exist an integer
$t \le \lceil 4/\eps^2\rceil$ and a decomposition
\[
   \tau = \sum_{j < F(t)}\lambda_j(\bu_{j,1}\otimes\cdots\otimes\bu_{j,k}) + R + Q,
\]
with unit vectors $(\bu_{j,\ell})_{j,\ell}$ and
$\lambda_1, \ldots, \lambda_{F(t)-1} \in \R$ and 
\[
   \Fnorm{R} < \eps\Fnorm{\tau}
   \qquad\text{and}\qquad
   \lambda_{\max}(Q) < 
\frac{\eps\Fnorm{\tau}}
   {\sqrt{F(t{+}1) - F(t)}}
\]

Moreover, each $(\lambda_j; \bu_{j,1}, \ldots, \bu_{j,k})$ may be
chosen to be a top eigenpair of the residual tensor
$\tau -\sum_{i=1}^{j-1}\lambda_i(\bu_{i,1}\otimes\cdots\otimes\bu_{i,k})$.
For a symmetric $\tau$ the factors may be chosen equal, i.e.\ $\bu_{j,\ell}=\bu_j$ so that the $j$-th rank-one tensor is
$\bu_j^{\otimes k}$. For $\K = \C$ and for
$\K = \R$ with odd $k$ the $\lambda_j$'s can  be
taken positive.
\end{theorem}

\begin{proof}
Set $T = \lceil 4/\eps^2\rceil$ and $N = F(T{+}1) - 1$.
As in the proof of Theorem~\ref{thm:reglem}, but without
any stopping rule, choose for $j = 1, \ldots, N$ a top
eigenpair $(\lambda_j; \bu_{j,1}, \ldots, \bu_{j,k})$ of the residual
$\tau_{j-1} = \tau - \sum_{i=1}^{j-1}\lambda_i(\bu_{i,1}\otimes\cdots\otimes\bu_{i,k})$,
so that
$|\lambda_j| = \lambda_{\max}(\tau_{j-1})$ and, after the
sign normalization of a single factor, every $\lambda_j$ is
a nonnegative real; if $\tau$ is symmetric the eigenpairs
may instead be chosen diagonal, $\bu_{j,\ell} = \bu_j$,
with $\lambda_j \in \R$, by the
last assertion of Theorem~\ref{thm:reglem}.
By Lemma~\ref{lem:Pythagoras},
\begin{equation}\label{eq:pigeonhole}
   \Fnorm{\tau_\ell}^2 = \Fnorm{\tau}^2
   - \sum_{j=1}^{\ell}\lambda_j^2
   \qquad (\ell \le N),
   \qquad\text{in particular}\qquad
   \sum_{j=1}^{N}\lambda_j^2 \le \Fnorm{\tau}^2 .
\end{equation}
The intervals $I_t = [F(t), F(t{+}1))$,
$t = 1, \ldots, T$, are disjoint subsets of $[1, N]$, so
for some $t \le T$,
\begin{align}\label{eq:pigeon}
   \sum_{j \in I_t}\lambda_j^2
   \le \frac{\Fnorm{\tau}^2}{T}
   \le \frac{\eps^2}{4}\Fnorm{\tau}^2 .
\end{align}
Let $m + 1$ be an index minimizing $|\lambda_j|$ over
$j \in I_t$. We  set
\[
   R = \sum_{j = F(t)}^{m}\lambda_j(\bu_{j,1}\otimes\cdots\otimes\bu_{j,k})
   = \tau_{F(t)-1} - \tau_m
   \qquad\text{and}\qquad
   Q = \tau_m,
\]
and by telescoping we obtain
\[\tau = \sum_{j < F(t)}\lambda_j(\bu_{j,1}\otimes\cdots\otimes\bu_{j,k}) + R + Q.\]
Since $\lambda_{\max}(Q)=|\lambda_{m+1}|$ we obtain by definition of $m+1$, by~\eqref{eq:pigeon} and $|I_t| = F(t{+}1) - F(t)$ that
\[
   \lambda_{\max}(Q) = |\lambda_{m+1}|
   \le \frac{\eps\Fnorm{\tau}}
   {2\sqrt{|I_t|}}
= \frac{\eps\Fnorm{\tau}}
   {2\sqrt{F(t{+}1) - F(t)}}.
\]
Concerning $\Fnorm{R}$ note that 
\[
   \Fnorm{R}^2=\Fnorm{ \tau_{F(t)-1} - \tau_m}^2
   = \Fnorm{\tau_{F(t)-1}}^2 - \Fnorm{\tau_m}^2
   - 2\mathrm{Re}\langle R, \tau_m\rangle_F
\]
and by the first property of~\eqref{eq:pigeonhole} and~\eqref{eq:pigeon} we have $\Fnorm{\tau_{F(t)-1}}^2 - \Fnorm{\tau_m}^2=\sum_{j=F(t)}^{m}\lambda_j^2\le \tfrac{\eps^2}{4}\Fnorm{\tau}^2$.
To bound the remaining term note that~\eqref{eq:evalFrob} gives for every $j$
\[|\langle (\bu_{j,1}\otimes\cdots\otimes\bu_{j,k}), \tau_m\rangle_F| =
|\tau_m(\conj\bu_{j,1}, \ldots, \conj\bu_{j,k})| \le
\lambda_{\max}(\tau_m) = \lambda_{\max}(Q),\]
and thus by Cauchy--Schwarz  and~\eqref{eq:pigeon} we have
\[
   \bigl|\langle R, \tau_m\rangle_F\bigr|
   \le \lambda_{\max}(Q)\sum_{j=F(t)}^{m}|\lambda_j|
   \le \frac{\eps\Fnorm{\tau}}{2\sqrt{|I_t|}}
   \cdot\sqrt{|I_t|}\sqrt{\sum_{j\in I_t}|\lambda_j|^2}
   \le \frac{\eps^2}{4}\Fnorm{\tau}^2.
\]
Hence
$\Fnorm{R}^2 \le \tfrac{3}{4}\eps^2\Fnorm{\tau}^2$ and
$\Fnorm{R} < \eps\Fnorm{\tau}$, as desired.
\end{proof}

\paragraph{Cayley-type tensors over finite abelian groups} 
Cayley graphs over a finite abelian group are diagonalized by the characters of the group. Chung and Graham~\cite{CG92} showed that such a graph is quasi-random if and only if its generating set has \emph{small bias}, meaning that all non-trivial Fourier coefficients of its indicator function are small, a condition which can be traced back to Roth~\cite{Roth53}.

Hypergraph tensors, on the other hand, do not admit  such diagonalization in general. Nevertheless, using the result of Lenz and Mubayi~\cite{LM1} we extended the result of Chung and Graham  to Cayley-type hypergraphs in~\cite{AH20}. We observe next that the corresponding tensors are diagonalized by rank-one tensors built from characters, and we briefly discuss what the spectral regularity lemma says in this case.

Let $G$ be a finite abelian group of order $n$ with dual group $\widehat G$ consisting of the characters of $G$, i.e., the homomorphisms $\eta \colon G \to \T$ into the multiplicative group $\T$ of complex numbers of modulus one. The characters $ \eta, \xi \in \widehat G$ satisfy the orthogonality relation
 $\langle\eta,\xi\rangle=  \sum_{x \in G}\eta(x)\conj\xi(x)=\delta_{\eta\xi}n$,
and $\{\eta/\sqrt n\}_{\eta\in\widehat G}$ is an
orthonormal basis of $\C^G$ with respect to the inner
product $\langle \bu, \bv\rangle =
\sum_{x}\bu(x)\conj{\bv(x)}$.
The \emph{Fourier coefficient} of $f \colon G \to \C$ at $\eta \in \widehat G$ is
$\widehat f(\eta) = \langle f, \eta\rangle= \sum_{x\in G}f(x)\conj\eta(x)$,
so that \[f = \tfrac1n\sum_{\eta\in\widehat G}\widehat f(\eta)\eta\qquad \text{and}\qquad
\tfrac1n\sum_{\eta\in\widehat G}|\widehat f(\eta)|^2 = \sum_{x\in G}|f(x)|^2 \quad\text{(Parseval).}\]

For an integer $k \ge 2$, a vector $\bc=(c_1,\dots, c_k) \in \Z^k$ is called \emph{admissible} if  each $c_i$ is coprime to  the group exponent (so that $x \mapsto c_ix$ and  $\eta \mapsto \eta^{c_i}$
are automorphisms of $G$ and of $\widehat G$, respectively). The \emph{Cayley $\bc$-tensor} of $f$ is
$\tau_f$ given by
$(\tau_f)_{x_1\cdots x_k} = f(c_1x_1 + \cdots + c_kx_k)$.

\begin{corollary}[Regularity lemma for
abelian groups with growth function]\label{cor:green-weak}
Let  $F \colon \N \to \N$ be a strictly increasing function. For
every $\eps > 0$ and $f \colon G \to \C$ there exist an integer
$t\le \lceil 4/\eps^2\rceil$ and  a set
$S \subseteq \widehat G$ with $|S| < F(t)$
 such that for every $k \ge 2$
and every admissible $\bc = (c_1, \ldots, c_k)$ the Cayley
$\bc$-tensor $\tau_f$ of $f$ decomposes as
\[
   \tau_f =\frac{1}{n} \sum_{\eta \in S}\widehat f(\eta)({\eta^{c_1}}\otimes\cdots\otimes{\eta^{c_k}}) + \tau_R + \tau_Q,
\]
with Cayley $\bc$-tensors $\tau_R$ and  $\tau_Q$ satisfying
\[
   \Fnorm{\tau_R} \le \eps\Fnorm{\tau_f}
   \qquad\text{and}\qquad
   \lambda_{\max}(\tau_Q) \le 
  \frac{\eps \Fnorm{\tau_f}}{ \sqrt{F(t{+}1) - F(t)}}.
\]
\end{corollary}

\begin{proof}
Since each $\eta \mapsto \eta^{c_\ell}$ is a bijection
of $\widehat G$ we  have for any $\eta,\xi\in\widehat G$ 
\begin{align}\label{eq:ortho} \bigl\langle({\eta^{c_1}}\otimes\cdots\otimes{\eta^{c_k}}), ({\xi^{c_1}}\otimes\cdots\otimes{\xi^{c_k}})\bigr\rangle_F =
\prod_\ell\bigl\langle{\eta^{c_\ell}}, {\xi^{c_\ell}}\bigr\rangle =\begin{cases} n^{k} &\text{if }  \eta=\xi\\ 0&\text{otherwise.}\end{cases}
\end{align}
Further, as $\eta\left(\sum c_ix_i\right)
= \prod_{\ell}\eta^{c_\ell}(x_\ell)$ we obtain by the Fourier
expansion of $f$ 
\begin{align*}
   (\tau_f)_{x_1\cdots x_k}
   &= f\Big(\sum_{i\in[k]} c_ix_i\Big)
   = \frac1n\sum_{\eta\in\widehat G}\widehat f(\eta)\eta \Big(\sum_{i\in[k]} c_ix_i\Big)
   = \frac1n\sum_{\eta\in\widehat G}\widehat f(\eta)
   \prod_{\ell\in[k]}\eta^{c_\ell}(x_\ell).
\end{align*}
As $\prod_{\ell}\eta^{c_\ell}(x_\ell)$ is the
$(x_1, \ldots, x_k)$-entry of
$\big(\eta^{c_1}\otimes\cdots\otimes\eta^{c_k}\big)$ we obtain
\[
   \tau_f = \frac1n\sum_{\eta\in\widehat G}\widehat f(\eta)
   \big(\eta^{c_1}\otimes\cdots\otimes\eta^{c_k}\big).
\]
In other words, $\tau_f$ is diagonalized by the  family  $\left\{\eta^{c_1}\otimes\cdots\otimes\eta^{c_k}\right\}_{\eta\in\widehat G}$ of pairwise Frobenius orthogonal rank-one tensors, each of Frobenius norm $n^{k/2}$ by~\eqref{eq:ortho}, and we also deduce 
$\Fnorm{\tau_f}^2 =  n^{k-2}\sum_\eta|\widehat f(\eta)|^2$. 

This diagonalization enables us to calculate the top eigenvalue $\lambda_{\max}(\tau_f)$ as follows.
We have
\[\big(\eta^{c_1}\otimes\cdots\otimes\eta^{c_k}\big)(\bv_1, \ldots, \bv_k)=\sum_{x_1\dots x_k} \prod_{\ell=1}^k\eta^{c_\ell}(x_\ell) \bv_\ell(x_\ell)= 
\prod_{\ell=1}^k\Big(\sum_x\eta^{c_\ell}(x)\bv_\ell(x)\Big)=\prod_{\ell=1}^k\big\langle \bv_\ell, \eta^{-c_\ell}\big\rangle=\prod_{\ell=1}^k\widehat{\bv}_\ell\big(\eta^{-c_\ell}\big).\]
Let $d^{(\ell)}_\eta=\widehat{\bv_\ell}(\eta^{-c_\ell})$ and $d^{(\ell)}=(d^{(\ell)}_\eta)_{\eta\in\widehat G}$. Since $\{\eta^{-c_\ell}\}_{\eta\in\widehat G}$ is an orthogonal basis of $\C^G$ Parseval yields
$\|d^{(\ell)}\|_2^2=\sum_\eta|d^{(\ell)}_\eta|^2= \sum_\eta \big|\widehat{\bv_\ell}(\eta)\big|^2= n\|\bv_\ell\|_2^2$ for each $\ell$. By H\"older we deduce for unit vectors $\bv_1,\dots,\bv_k$  
\[
|\tau_f(\bv_1, \ldots, \bv_k)|=\frac1n\Bigl|\sum_\eta \widehat f(\eta)
   d^{(1)}_\eta\cdots d^{(k)}_\eta\Bigr|\leq \frac1n\max_\eta|\widehat f(\eta)|   \prod_{\ell}\|d^{(\ell)}\|_k
   \le n^{k/2-1}\max_\eta|\widehat f(\eta)|, 
\]
where we used $\|d^{(\ell)}\|_k \le \|d^{(\ell)}\|_2 = \sqrt n\|\bv_\ell\|_2 = \sqrt n$
for $k \ge 2$.
This shows that 
\[
   \lambda_{\max}(\tau_f) = n^{k/2-1}\max_\eta|\widehat f(\eta)|
\]
where the lower bound follows by evaluation at
$(\conj{\eta^{c_1}}/\sqrt n, \ldots, \conj{\eta^{c_k}}/\sqrt n)$.

Now fix an enumeration $\eta_1, \eta_2, \ldots$ of $\widehat G$
with $|\widehat f(\eta_1)| \ge |\widehat f(\eta_2)| \ge \cdots|\widehat f(\eta_n)|$ (which is independent of $k$ and $\bc$)
and apply Theorem~\ref{thm:tao-reglem} over $\K = \C$ to
$\tau_f$ with the growth function $F$. 
At step $j$ the top eigenpair is $(n^{k/2-1}\widehat f(\eta_j);\ \eta_j^{c_1}/\sqrt n,\dots,\eta_j^{c_k}/\sqrt n)$ and the decomposition follows
the  truncation along the  enumeration from above. 
We also obtain $t \le \lceil 4/\eps^2\rceil$ and  $m \in [F(t), F(t{+}1))$, and an inspection of its
proof shows that $t$ and $m$ depend only on the sequence
of ratios
$|\widehat f(\eta_j)|^2/\sum_\eta|\widehat f(\eta)|^2$, on $\eps$
and on $F$. 
Set
$S = \{\eta_j \colon j < F(t)\}$,
$R = \{\eta_j \colon F(t) \le j \le m\}$ and
$Q = \widehat G \setminus (S \cup R)$, and for
$M \subseteq \widehat G$ write
\[
   f_M = \tfrac1n\sum_{\eta\in M}\widehat f(\eta)\eta
   \qquad\text{and}\qquad
   \tau_M = \tau_{f_M}
\]
for the Cayley $\bc$-tensor of $f_M$. Then
$f = f_S + f_R + f_Q$ and the decomposition
$\tau_f = \tau_S + \tau_R + \tau_Q$  holds for
every $k$ and every admissible $\bc$ simultaneously and with  structured part 
$\tau_S = \tfrac1n\sum_{\eta\in S}\widehat f(\eta)(\eta^{c_1}\otimes\cdots\otimes\eta^{c_k})$.
The bounds on  $\Fnorm{\tau_R}$  and $\lambda_{\max}(\tau_Q)$follow directly from Theorem~\ref{thm:tao-reglem}.
\end{proof}

\begin{remark}
\label{rem:fourier-form}
From the proof also immediately get that $\Fnorm{\tau_R}^2 = n^{k-2}\sum_{\eta\in R}|\widehat f(\eta)|^2$,
$\lambda_{\max}(\tau_Q) = n^{k/2-1}\sup_{\eta\in Q}|\widehat f(\eta)|$
and
$\Fnorm{\tau_f}^2 = n^{k-2}\sum_\eta|\widehat f(\eta)|^2$,
so in terms of the Fourier coefficients the two bounds  on  $\Fnorm{\tau_R}$  and $\lambda_{\max}(\tau_Q)$ read
\[
   \sum_{\eta\in R}\bigl|\widehat f(\eta)\bigr|^2
   \le \frac{\eps^2}4
   \sum_{\eta\in\widehat G}\bigl|\widehat f(\eta)\bigr|^2
   \qquad\text{and}\qquad
   \sup_{\eta\in Q}\bigl|\widehat f(\eta)\bigr|^2
   \le
   \left(\frac{\eps}{F(t{+}1) - F(t)}\right)
   \sum_{\eta\in\widehat G}\bigl|\widehat f(\eta)\bigr|^2.
\]
\end{remark}

Let us briefly discuss what happens when we apply the partition mechanism analogous to Lemma~\ref{lem:partition} to the 
structured part
$\tau_S = n^{-1}\sum_{\eta\in S}\widehat f(\eta)
\bigl({\eta^{c_1}}\otimes\cdots\otimes{\eta^{c_k}}\bigr)$ of a Cayley $\bc$-tensor $\tau_f$
and how this relates to Green's regularity lemma for subsets of finite abelian groups~\cite{Green05}. As
 characters have all entries on the
unit circle  $\T$ the junk set is empty for $C\geq 1$ and it is natural to partition~$\T$ into \emph{arcs}
instead of the $\alpha$-approximate level sets as in Lemma~\ref{lem:partition}. Explicitly, let $t = \lceil 1/\alpha\rceil$ and for $j = 0, 1, \dots, t-1$ consider the intervals
\[
   I_j = \bigl[\tfrac{2j-1}{2t}, \tfrac{2j+1}{2t}\bigr)
   \qquad\text{and the arcs}\quad
   e(I_j) = \bigl\{e^{2\pi i\theta}\colon \theta \in I_j\bigr\}\subset\T
\]
of angular width $1/t \le \alpha$ and centered at $e^{2\pi ij/t}$. 
For each $\xi\in S^{\bc}=\{\eta^{c_\ell}\colon \eta \in S, \ell\in[k]\}$ these arcs induce 
a partition of $G$ into  sets
$\bigl\{x\in G\colon \xi(x)\in e(I_j)\bigr\}$ and
the joint refinement then yields the partition~$\cP$ of $G$, where within a
cell of $\cP$ every character in $S^{\bc}$ varies by at most the arc
length~$2\pi/t \le 2\pi\alpha$.
Analogous calculation as in
Lemma~\ref{lem:partition} then yield that $\tau_f$ is
approximated by its block-averages over these cells.

As an example consider $G = \mathbb{F}_p^n$. Then each
character  $\xi\in S^{\bc}$ takes only the values $e(2\pi ij/p)$, $j=0,\dots,p-1$, hence for $\alpha \le 1/p$ the partition induced by $\xi$ correspond to the $p$ cosets of its kernel.
The joint refinement
are the cosets of the subgroup
$H = \bigcap_{\xi \in S^{\bc}}\ker\xi
= \bigcap_{\eta\in S}\ker\eta$, 
a subgroup
of index at most $p^{|S|}$. Thus, the partition $\cP$ is a partition by
cosets of $H$.

For  general $G$ note that the
$0$-th arc is centered at $1$ and  the induced cell is  the \emph{Bohr set}
\[
   B(S^{\bc}, \beta) =
   \bigl\{x \in G\colon \eta^{c_\ell}(x) \in
   e\bigl([-\beta, \beta)\bigr)
   \ \text{for all}\ \eta \in S,\ \ell\in[k]\bigr\},
   \qquad \beta = \tfrac{1}{2t}.
\]
By a result of Bourgain~\cite{Bourgain} we may also slightly modify $\beta$  to $\beta/2\leq\beta^*\leq\beta$ and make
$B(S^{\bc}, \beta^*)$ a \emph{regular} Bohr set, meaning roughly that its size is stable under small changes of the radius. We refer to~\cite{Bourgain} for more details.
The remaining cells are its companion
\emph{phase fibres}: for $x, x'$ in a common cell,
$\xi(x - x') \in e([-2\beta, 2\beta])$ for all
$\xi \in S^{\bc}$, so every nonempty cell has difference
set contained in $B(S^{\bc}, 2\beta)$ and in particular
lies in a translate of it.  

Green's regularity lemma~\cite{Green05} is stronger as it
bounds the Fourier coefficients of the quasi-random part
$f_Q$ {relative to almost every cell} (respectively,
translate), not only globally as in
Corollary~\ref{cor:green-weak}. 
It would be interesting to see whether or not his result
can be (easily)
recovered from our approach, but we do not pursue this here.

\paragraph{Acknowledgement and declaration of AI use} This project was initiated in 2024, motivated  by  the author's  previous works on the topic~\cite{CHPS12,ACHPS18,AH20} 
and by his supervision of the  thesis of Felipe Sanchez Erazo on (spectral) regularity lemma for graphs~\cite{Felipe}. Many thanks to Felipe for his help initiating  the project.
 
This work was completed in  dialogue with Claude (Anthropic), which was used for many exploratory computations, for drafting and reading proofs,  literature search,  and bibliography verification. The author posed the problems, chose the lines of attack, supplied the key ideas, and has verified all statements and proofs.

\bibliographystyle{plain}

\end{document}